\documentclass[12pt]{article}
\usepackage{latexsym}
\usepackage{amscd}
\usepackage{amssymb}
\DeclareMathAlphabet{\eusm}{OT1}{eusm}{m}{n}

\newtheorem{theor}{Theorem}[section]
\newtheorem{prop}[theor]{Proposition}
\newtheorem{defi}[theor]{Definition}
\newtheorem{cor}[theor]{Corollary}
\newtheorem{lem}[theor]{Lemma}
\newtheorem{ex}[theor]{Example}
\newtheorem{rem}[theor]{Remark}
\def\Ann{\mbox{Ann\/}}

\newenvironment{proof}{\par\noindent{\it Proof.}}{$\Box$\par\bigskip}

\begin{document}

\title {On Two Classes of Modules Related to CS Trivial Extensions}
\author{}

\author{Farid Kourki \\
Centre R\'{e}gional des M\'{e}tiers de l'Education\\ et de la Formation (CRMEF)-Tanger,\\
Annexe de Larache,
B.P. 4063, Larache, Morocco\\
\small{kourkifarid@hotmail.com}\\
Rachid Tribak\\
Centre R\'{e}gional des M\'{e}tiers de l'Education\\ et de la Formation (CRMEF)-Tanger,\\
Avenue My Abdelaziz, Souani, B.P. 3117,\\ Tanger, Morocco\\
\small{tribak12@yahoo.com}}

\date{}

\maketitle

\begin{abstract}
\renewcommand{\thefootnote}{\fnsymbol{footnote}}
\setcounter{footnote}{-1}
\renewcommand{\thefootnote}{\arabic{footnote}}
All rings considered are commutative. In this article we introduce and study two notions of modules which are stronger than CS modules,
namely weakly IN modules and strongly CS modules.
Our main aim is to characterize when a trivial extension $A=R\propto M$ (of a ring $R$
by an $R$-module $M$) is a CS ring.
\end{abstract}

\maketitle

\section{Introduction}
Throughout this article, all rings considered are assumed to be commutative rings with an identity and $R$ denotes such a ring. All modules are unital.
We denote respectively by $Spec(R)$, $Max(R)$ and $Min(R)$ the set of all prime ideals of $R$, the set of all maximal ideals of $R$
and the set of all minimal prime ideals of $R$. The nilradical of $R$ and the Jacobson radical of $R$ are denoted by $Nil(R)$ and $J(R)$, respectively.
Let $M$ be an $R$-module and let $x \in M$. By $Ann_R(x)$ and $Ann_R(M)$ we denote the {\it annihilator} of $x$ and $M$, respectively;
i.e. $Ann_R(x)=\{r\in R \mid rx=0\}$ and $Ann_R(M)=\{r\in R \mid rM=0\}$. The notation $N \subseteq M$ means that $N$ is a subset of $M$;
$N \leq M$ means that $N$ is a submodule of $M$; and $N \subseteq ^{ess} M$ means that $N$ is an essential submodule of $M$.
By $\mathbb{Q}$ and $\mathbb{Z}$ we denote the ring of rational and integer numbers, respectively.

An $R$-module $M$ is called CS (or extending) if every submodule of $M$ is essential in a direct summand of $M$ and a ring $R$ is called CS
if the $R$-module $R$ is CS.
By \cite[Theorem 6]{Kourki:2003}, $R$ is CS if and only if for any ideals $I$ and $J$ of $R$ with $I\cap J=0$, $Ann_R(I)+Ann_R(J)=R$. In our attempt to characterize CS trivial extensions, we were lead to introduce two types of CS modules. The first one comes form the above characterization of CS rings.
For consider $A=R\propto M$, the trivial extension of $R$ by an $R$-module $M$, and let $\phi:M \rightarrow A$ be the natural monomorphism.
Assume that $A$ is a CS ring. Then given two submodules $N$ and $L$ of $M$ such that $N\cap L=0$, it is clear that $\phi(N)$ and $\phi(L)$ are two ideals of $A$ such that $\phi(N)\cap\phi(L)=0$. Thus $Ann_A\phi(N)+Ann_A\phi(L)=A$ which implies that $Ann_R(N)+Ann_R(L)=R$.
Recall that a ring $R$ is called an Ikeda-Nakayama ring (or IN ring) if for any two ideals $T$ and $T^{\prime}$ of $R$, we have
$Ann_R(T \cap T^{\prime})=Ann_R(T)+Ann_R(T^{\prime})$ (see \cite[p. 148]{NY}).
Call an $R$-module $M$ weakly IN if $Ann_R(N)+Ann_R(L)=R$ whenever $N$ and $L$ are submodules of $M$ such that $N\cap L=0$. The second type is due to the definition itself of CS rings. Suppose that $A$ is a CS ring and let $N$ be a submodule of $M$. Then $\phi(N)$ is an ideal of $A$ and hence $\phi(N)$ is essential in $fA$ for some idempotent $f$ of $A$. But $f$ is of the form $(e,0)$ for some idempotent $e$ of $R$, so $N$ is essential in $eM$. Call an $R$-module $M$ strongly CS if for every submodule $N$ of $M$, there exists an idempotent $e$ of $R$ such that $N$ is an essential submodule of $eM$. The above discussion justifies our choice of the title and motivates our study in this paper.

In Section 2, we study some properties of weakly IN modules and strongly CS modules. We investigate the connection between these two classes
of modules. It is shown that an $R$-module $M$ is strongly CS if and only if $M$ is weakly IN and idempotents lift modulo $Ann_R(M)$ (Theorem \ref{WIN}).
Then we characterize the class of rings $R$ over which every cyclic $R$-module is weakly IN (strongly CS)
(Theorems \ref{cyclics are weakly IN} and \ref{cyclics are strongly CS}).

In section 3, we focus on when a finite direct sum of weakly IN (strongly CS) modules is also weakly IN (strongly CS).
Let an $R$-module $M=M_1\oplus M_2\oplus \cdots \oplus M_n$ be a direct sum of submodules $M_1, M_2, \dots, M_n$.
We show that $M$ is weakly IN if and only if each $M_i$ is weakly IN and $Ann_R(M_i)+Ann_R(M_j)=R$ for all distinct $i$, $j$ in $\{1, \ldots, n\}$
(Theorem \ref{weakly IN direct sum}).
It is also shown that $M$ is a strongly CS $R$-module if and only if $R=R_1\times R_2\times \cdots \times R_n$ such that $M_i$ is a strongly CS $R_i$-module for all $i\in \{1,\dots ,n\}$ (Theorem \ref{finite strongly CS direct sum}). As an application, we prove that $R$ is a clean ring (i.e. every element of $R$ is a sum of a unit and an idempotent) if and only if every weakly IN $R$-module is strongly CS if and only if $R/\mathfrak{m}\oplus R/\mathfrak{m}'$ is a strongly CS $R$-module for every distinct maximal ideals $\mathfrak{m}$ and $\mathfrak{m}'$ of $R$ (Theorem \ref{clean rings}). Also, we characterize the class of rings $R$ for which $R/\mathfrak{p}\oplus R/\mathfrak{p}'$ is a strongly CS $R$-module for every distinct minimal prime ideals $\mathfrak{p}$ and $\mathfrak{p}'$ of $R$ as that of the purified rings (Theorem \ref{purified-characterization}).
Recall that a ring $R$ is said to be a purified (or coclean) ring if for every distinct minimal prime ideals $\mathfrak{p}$ and $\mathfrak{q}$ of $R$,
there exists an idempotent $e$ of $R$ such that $e\in \mathfrak{p}$ and $1-e\in \mathfrak{q}$.

In section 4, we characterize weakly IN and strongly CS modules over Dedekind domains. In particular, we show that if $R$ is a Dedekind domain and $M$ is an $R$-module, then $M$ is weakly IN if and only if $M$ is cyclic, or $M \cong E(R/\mathfrak{m})$ for some maximal ideal $\mathfrak{m}$ of $R$,
or $M$ is isomorphic to an $R$-submodule of the quotient field $K$ of $R$ (Theorem \ref{WIN over Dedekind domains}).

The last section is devoted to the study of CS trivial extensions. Let $A=R\propto M$ denotes the trivial extension of a ring $R$ by an $R$-module $M$.
In the main result of this section, we show that $A$ is a CS ring if and only if $M$ is weakly IN (or strongly CS) and there exists
an idempotent $e$ of $R$ such that $Ann_R(M)=eR$ and $eR$ is a CS ring if and only if $M\oplus Ann_R(M)$ is a weakly IN (or strongly CS) $R$-module
(Theorem \ref{CS TX}). When $M$ is a flat $R$-module, we prove that $A$ is a CS ring if and only if $R$ and $M$ are weakly IN (or strongly CS) $R$-modules
and $Ann_R(M)$ is a direct summand of $R$ (Corollary \ref{flat}).

\section{Weakly IN and Strongly CS Modules}

Recall that an $R$-module $M$ is called CS if every submodule of $M$ is essential in a direct summand of $M$.
A ring $R$ is called CS if $R$ is a CS $R$-module. Using \cite[Theorem 6]{Kourki:2003}, we obtain the following proposition.

\begin{prop}\label{CS rings} The following are equivalent for a ring $R$:

\noindent{\em (1)} $R$ is a CS ring;

\noindent{\em (2)} For any ideal $I$ of $R$, there exists $e=e^2 \in R$ such that $I\subseteq^{ess}eR$;

\noindent{\em (3)} For any ideals $I$ and $J$ of $R$ with $I\cap J=0$, $Ann_R(I)+Ann_R(J)=R$.
\end{prop}

A ring $R$ is called an Ikeda-Nakayama ring (or an IN-ring) if $Ann_R(I\cap J)=Ann_R(I)+Ann_R(J)$ for all ideals $I$ and $J$ of $R$ (see \cite{CNY}).
Using the endomorphism ring, Wisbauer, Yousif and Zhou generalized this notion to a module theoretic version in 2002
\cite{Wisbrauer et Yousif et Zhou :1994}.
Here, we will consider another generalization. We will call an $R$-module $M$ an {\it s.IN-module} ({\it scalar IN-module}) if
$Ann_R(N\cap L)=Ann_R(N)+Ann_R(L)$ for all submodules $N$ an $L$ of $M$. Next, we introduce two notions.
The first one is weaker than that of s.IN-modules and the latter one is stronger than that of CS modules.

\begin{defi} \rm Let $M$ be an $R$-module.

\noindent{(1)} $M$ is called {\it weakly IN} if for any submodules $N$ and $L$ of $M$ with $N\cap L=0$, $Ann_R(N)+Ann_R(L)=R$.

\noindent{(2)} $M$ is called {\it strongly CS} if for any submodule $N$ of $M$, there exists $e=e^2 \in R$ such that $N\subseteq^{ess}eM$.

\noindent{(3)} A ring $R$ is called weakly IN (strongly CS) if $R$, as an $R$-module, has the corresponding property.
\end{defi}

\begin{rem}\label{submodules of WIN modules}\em (1) From Proposition \ref{CS rings}, it follows easily that a ring $R$ is CS
if and only if $R$ is weakly IN if and only if $R$ is strongly CS.

(2) Consider the ring $\mathbb{Z}_2[x_1, x_2, \ldots]$ where $x_i^3=0$ for all $i$, $x_ix_j=0$ for all $i \neq j$ and $x_i^2=x_j^2 \neq 0$
for all $i$ and $j$. It was shown in \cite[Example 6]{CNY} that $R$ is a CS ring and hence $R$ is weakly IN, but $R$ is not an IN ring.
It follows that the $R$-module $R$ is weakly IN, but it is not an s.IN-module.

(3) If $M$ is a uniform $R$-module, then clearly $M$ is strongly CS. If $R$ is indecomposable or $M$ is indecomposable,
then the converse also holds.
\end{rem}

\begin{prop} \label{closed-submodules} Any submodule of a strongly CS $($weakly IN$)$ module is strongly CS $($weakly IN$)$.
\end{prop}

\begin{proof} Let $M$ be a strongly CS module and let $N$ be a submodule of $M$. Let $L$ be a submodule of $N$.
Then $L \subseteq^{ess}eM$ for some idempotent $e$ of $R$. Hence $L \subseteq^{ess}eM \cap N = eN$.
The other assertion is evident.
\end{proof}

\begin{prop}\label{free weakly IN} A free $R$-module $F$ is weakly IN if and only if $rank(F)=1$ and $R$ is a CS ring.
\end{prop}

\begin{proof} Let $F$ be a weakly IN free $R$-module. Suppose that $rank(F)\geq 2$. Then $F$ contains a submodule isomorphic to $R\oplus R$.
Therefore $Ann_R(R)+Ann_R(R)=R$, a contradiction. So $rank(F)=1$ and hence $R$ is a CS ring (Remark \ref{submodules of WIN modules}(1)). The converse is clear.
\end{proof}

From the preceding proposition, one can directly infer that over a field $K$, a module $M$ is weakly IN if and only if
$M \cong K$.

Recall that a module $M$ is called {\it quasi-continuous} if $M$ is CS and, for any two direct summands $M_1$, $M_2$ with $M_1 \cap M_2 = 0$,
$M_1 \oplus M_2$ is also a direct summand (see \cite{MM}). From \cite[Corollary 4]{Wisbrauer et Yousif et Zhou :1994},
it follows that a faithful $R$-module $M$ is weakly IN if and only if $M$ is a quasi-continuous $R$-module and for any  $f^2=f\in End_R(M)$,
there exists $r\in R$ such that $f(x)=rx$ for any $x\in M$. This fact can be extended to the general case as shown in the following characterization
of weakly IN modules.

\begin{prop}\label{pseudo-Baer modules}
Let $R$ be a ring and let $M$ be a nonzero $R$-module. Then the following are equivalent:

\noindent{\em (1)} $M$ is a weakly IN $R$-module;

\noindent{\em (2)} $M$ is a weakly IN $R/Ann_R(M)$-module;

\noindent{\em (3)} $M$ is a quasi-continuous $R$-module and for any  $f^2=f\in End_R(M)$, there exists $r\in R$ such that $f(x)=rx$ for any $x\in M$;

\noindent{\em (4)} $M$ is a CS $R$-module and for any direct summand $N$ of $M$, there exists $r\in R$ such that $N=rM$ and $r-r^2\in Ann_R(M)$;

\noindent{\em (5)} For any submodule $N$ of $M$, there exists  $r\in R$ such that $N\subseteq^{ess}rM$  and $r-r^2\in Ann_R(M)$;

\noindent{\em (6)} $M$ is strongly CS as an $R/Ann_R(M)$-module.
\end{prop}

\begin{proof} (1) $\Leftrightarrow$ (2) This is clear.

(2) $\Rightarrow$ (3) Let $\overline{R}=R/Ann_RM$. It is easy to see that $M$ has a natural structure of ($\overline{R}$,$R$)-bimodule and $M$ is a faithful left $\overline{R}$-module. Moreover, the $R$-submodules and the $\overline{R}$-submodules of $M$ are the same.
Using \cite[Corollary 4]{Wisbrauer et Yousif et Zhou :1994}, we see that $M$ is a quasi-continuous $R$-module and for any $f^2=f\in End_R(M)$,
there exists $\overline{r}=r+Ann_R(M)\in \overline{R}$ such that $f(x)=\overline{r}x=rx$ for all $x\in M$.

(3) $\Rightarrow$ (4) It is clear that $M$ is a CS $R$-module. Let $N$ and $K$ be two submodules of $M$
such that $M=N\oplus K$. Let $f$ be the projection on $N$ along $K$. Then $f^2=f$ and so there exists $r\in R$
such that $f(x)=rx$ for any $x\in M$. Hence, $N=f(M)=rM$. But $f^2=f$, so $r-r^2\in Ann_R(M)$.

(4) $\Rightarrow$ (5) Let $N$ be a submodule of $M$. Since $M$ is CS, there exists a direct summand $K$ of $M$ such that $N\subseteq^{ess}K$.
Moreover, by assumption, $K=rM$ for some $r\in R$ with $r-r^2\in Ann_R(M)$.

(5) $\Rightarrow$ (1) Let $N$ and $L$ be two submodules of  $M$ with $N\cap L=0$. Then there exist $r$ and $s$ in $R$ such that $N\subseteq^{ess}rM$, $L\subseteq^{ess}sM$, $r-r^2\in Ann_R(M)$ and $s-s^2\in Ann_R(M)$. Thus $N \cap L \subseteq^{ess} rM \cap sM$.
Since $N\cap L=0$, we have $rM\cap sM=0$. This implies that $rsM\subseteq rM\cap sM=0$ and hence $rs\in Ann_R(M)$.
Moreover, since $r-r^2\in Ann_R(M)$ and $s-s^2\in Ann_R(M)$, we have $(1-r)\in Ann_R(rM)\subseteq Ann_R(N)$ and  $(1-s)\in Ann_R(sM) \subseteq Ann_R(L)$.
It follows that $$1=(1-r)+(1-s)r+sr\in Ann_R(N)+Ann_R(L).$$ Therefore, $R=Ann_R(N)+Ann_R(L)$. Thus, $M$ is weakly IN.

(5) $\Leftrightarrow$ (6) This is immediate.
\end{proof}

In the next example, we present a CS module which is not weakly IN, and some CS modules which are not strongly CS.
More examples are provided in the next two sections.

\begin{ex} \rm (1) Let $R$ be a self injective ring. Then clearly the $R$-module $M=R \oplus R$ is CS. However, $M$ is not weakly IN
by Proposition \ref{free weakly IN}.

(2) Let $R$ be an indecomposable ring (e.g., $R$ is a local ring or a domain).

(a) Consider the $R$-module $M=S_1 \oplus S_2$ where $S_1$ and $S_2$ are simple modules. It is clear that $M$ is a CS module.
On the other hand, $M$ is not strongly CS (see Remark \ref{submodules of WIN modules}(3)).

(b) Let $M$ be an injective $R$-module which is not indecomposable (e.g., we can take the $\mathbb{Z}$-module $M=\mathbb{Q} \oplus \mathbb{Q}$
or $M=\mathbb{Z}(p_1^{\infty}) \oplus \mathbb{Z}(p_2^{\infty})$ for some prime numbers $p_1$ and $p_2$). Then $M$ is CS, but $M$ is not
strongly CS by Remark \ref{submodules of WIN modules}(3).
\end{ex}

Let $R$ be a ring and let $I$ be a proper ideal of $R$. We say that idempotents lift modulo $I$ if whenever $r$ is an element of $R$ such that $r-r^2\in I$, then there exists $e=e^2\in R$ such that $r-e\in I$.

\begin{theor} \label{WIN} Let $R$ be a ring and let $M$ be a nonzero $R$-module. Then the following are equivalent:

\noindent{\em (1)} $M$ is a strongly CS $R$-module;

\noindent{\em (2)} For any submodules $N$ and $L$ of $M$ with $N\cap L=0$, there exists an idempotent $e\in R$ such that $e\in Ann_R(N)$ and $1-e\in Ann_R(L)$;

\noindent{\em (3)} $M$ is a weakly IN $R$-module and idempotents lift modulo $Ann_R(M)$;

\noindent{\em (4)} $M$ is a CS module and for every direct summand $K$ of $M$, there exists an idempotent $e$ of $R$ such that $K=eM$.
\end{theor}

\begin{proof} (1) $\Rightarrow$ (2)  Let $N$ and $L$ be two submodules of  $M$ with $N\cap L=0$. Then there exist $e=e^2$ and $f=f^2$ in $R$ such that $N\subseteq^{ess}fM$ and $L\subseteq^{ess}eM$. Using the fact that $N\cap L=0$, we get $fM\cap eM=0$. Hence, $efM \subseteq fM\cap eM=0$. So $e\in Ann_R(fM)\subseteq Ann_R(N)$ and $1-e\in Ann_R(eM)\subseteq Ann_R(L)$.

(2) $\Rightarrow$ (3)  Let $N$ and $L$ be two submodules of  $M$ with $N\cap L=0$. By hypothesis, there exists an idempotent $e\in R$
such that $e\in Ann_R(N)$ and $1-e\in Ann_R(L)$. But $e+1-e=1$, so  $Ann_R(N)+Ann_R(L)=R$. It follows that $M$ is weakly IN.
Now let $r\in R$ such that $r-r^2\in Ann_R(M)$ and let $rx=(1-r)y\in rM\cap (1-r)M$ where $x,y\in M$. Then $r^2x=(r-r^2)y=0$ since $r-r^2\in Ann_R(M)$. Using once again the fact that $r-r^2\in Ann_R(M)$, we get $rx=0$.  Therefore $rM\cap (1-r)M=0$. Hence, by (2), there exists $f=f^2\in R$
such that $1-f\in Ann_R(rM)$ and $f\in Ann_R((1-r)M)$. This means that $r-fr\in Ann_R(M)$ and $f-fr\in Ann_R(M)$.
Consequently, $$f-r=(f-fr)-(r-fr)\in Ann_R(M).$$ This shows that idempotents lift modulo $Ann_R(M)$.

(3) $\Rightarrow$ (1) Let $N$ be a submodule of $M$. Since $M$ is a weakly IN $R$-module, there exists $r\in R$ such that $N\subseteq^{ess}rM$
and $r-r^2\in Ann_R(M)$ (see Proposition \ref{pseudo-Baer modules}). But idempotents lift modulo $Ann_R(M)$, so there exits an idempotent $e$ of $R$
such that $r-e\in Ann_R(M)$. Therefore $rM=eM$ and consequently $N\subseteq^{ess}eM$.

(1) $\Leftrightarrow$ (4) This is clear.
\end{proof}

The preceding theorem shows that the class of weakly IN modules contains that of strongly CS modules. Next, we present an example
illustrating that this inclusion is proper, in general.

\begin{ex} \label{weakly IN-not-strongly CS} \rm Let $p$ and $q$ be two different prime numbers and consider the ring
$R=\{m/n \in \mathbb{Q} \mid p \nmid n$ and $q \nmid n \ (m/n$ in lowest terms$)\}$.
It is well known that $pR$ and $qR$ are the only maximal ideals in $R$. Moreover, idempotents do not lift modulo $J(R)=pR \cap qR$.
Let $M=R/pR \oplus R/qR$. Clearly, $Ann_R(M)=J(R)$. From Corollary \ref{prime ideals}, we infer that $M$ is a weakly IN $R$-module which not strongly CS.
\end{ex}

Let $R$ be a ring. Recall that the socle of $R$, denoted by $Soc(R)$, is the sum of all its minimal ideals.
In the following corollary, we provide sufficient conditions for a weakly IN module to be strongly CS.

\begin{cor}\label{Ann(R) direct summand} Let $M$ be a nonzero $R$-module such that $Ann_R(M)$ satisfies any one of the following conditions:

 \noindent{\em (1)}  $Ann_R(M)$ is a nil ideal of $R$ $($i.e., $Ann_R(M)\subseteq Nil(R)$$)$;

\noindent{\em (2)}  $Ann_R(M)$ is a direct summand of $R$ $($for instance, $M$ is faithful$)$;

\noindent{\em (3)} $Ann_R(M)=Soc(R)$.

\noindent Then $M$ is a weakly IN $R$-module if and only if $M$ is strongly CS.
\end{cor}

\begin{proof} The sufficiency follows from Theorem \ref{WIN}. Conversely, suppose that $M$ is a weakly IN $R$-module.
Applying again Theorem \ref{WIN}, we only need to show that idempotents lift modulo $Ann_R(M)$.

(1) This follows from \cite[Proposition 27.1]{AF}.

(2) We will show that idempotent lift modulo every direct summand of $R$. Let $e=e^2\in R$ and let $r\in R$ such that $r-r^2\in eR$.
Then $(1-e)(r-r^2)=0$ and hence $((1-e)r)^2=(1-e)r$. Thus $(1-e)r$ is an idempotent of $R$. Moreover, we have $r-((1-e)r)=er\in eR$.

(3) This follows from \cite[Lemma 1.2]{YoZh}.
\end{proof}

Let $n\geq 2$. By Proposition \ref{free weakly IN}, there exists no ring $R$ for which every $n$-generated $R$-module is weakly IN.
This fact should be contrasted with \cite[Corollary 13.8]{DHSW}. On the other hand, take a valuation ring $R$.
It is easily seen that every cyclic $R$-module is uniform. Therefore every cyclic $R$-module is strongly CS and weakly IN
(see Remark \ref{submodules of WIN modules}(3) and Theorem \ref{WIN}). Next, we will characterize the class of rings $R$ for which every cyclic $R$-module is weakly IN (strongly CS).

A ring $R$ is called a {\it CF-ring} if every $R$--module $M$ which is a direct sum of finitely many cyclic $R$--modules has
a {\it canonical form decomposition}, i.e. $M \cong R/I_1\oplus \cdots \oplus R/I_n$ where the ideals $I_j$ $(1 \leq j \leq n)$ of $R$ satisfy
$I_1\subseteq I_2\subseteq \cdots \subseteq I_n \varsubsetneq R$ (see \cite{SW}).
For the definitions of the other types of rings used in the following two results, we refer the reader to \cite{Brandal}.

\begin{theor}\label{cyclics are weakly IN} The following are equivalent for a ring $R$:

\noindent{\em (1)} Any cyclic $R$-module is weakly IN;

\noindent{\em (2)} $R/I$ is a CS ring for every proper ideal $I$ of $R$;

\noindent{\em (3)} $R$ is a CF-ring;

\noindent{\em (4)} $R$ is a finite direct product of valuation rings, h-local Pr\"{u}fer domains and torch rings.
\end{theor}

\begin{proof} Let $I$ be a proper ideal of $R$ and consider the $R$-module $M=R/I$. By Proposition \ref{pseudo-Baer modules},
$M$ is a weakly IN $R$-module if and only if $M$ is a strongly CS $R/Ann_R(M)$-module. This is equivalent to the condition that $R/I$ is a CS ring
(Remark \ref{submodules of WIN modules}(1)). Now use \cite[Theorem 9]{Kourki:2003}.
\end{proof}

Recall that a ring $R$ is called {\it clean} if every element of $R$ is a sum of a unit and an idempotent.

\begin{theor}\label{cyclics are strongly CS} The following are equivalent for a ring $R$:

\noindent{\em (1)} Any cyclic $R$-module is strongly CS;

\noindent{\em (2)} $R$ is a clean CF-ring;

\noindent{\em (3)} $R$ is a finite direct product of valuation rings.
\end{theor}

\begin{proof} (1) $\Leftrightarrow$ (2) From Theorem \ref{WIN}, it follows that the assertion (1) is equivalent to the condition that
for any ideal $I$ of $R$, $R/I$ is a weakly IN $R$-module and idempotents lift modulo $Ann_R(R/I)=I$. Now using Theorem \ref{cyclics are weakly IN}
and the fact that $R$ is a clean ring if and only if idempotents lift modulo every ideal of $R$ (see \cite[Theorem 5.1]{AT}),
we obtain the desired equivalence.

(2) $\Rightarrow$ (3)  Since $R$ is a CF-ring, we have $R=R_1\times R_2\times \cdots \times R_n$ where each $R_i$ is an indecomposable ring which is either
a valuation ring or an h-local Pr\"{u}fer domain or a torch ring (\cite[Theorems 3.10 and 3.12]{SW}). Let $i\in \{1,\dots ,n\}$. Since $R$ is clean,
it follows that $R_i$ is also clean  by \cite[Proposition 2]{AC}. Moreover, $R_i$ is a local ring by \cite[Theorem 3]{AC}.
Hence $R_i$ could not be a torch ring since every torch ring has at least two maximal ideals (\cite[page 38]{Brandal}).
In addition, note that any local Pr\"{u}fer domain is a valuation ring.

(3) $\Rightarrow$ (2) By Theorem \ref{cyclics are weakly IN}, $R$ is a CF-ring. Moreover, note that any valuation ring is local.
Thus, using \cite[Proposition 2(1)-(3)]{AC}, we conclude that $R$ is a clean ring.
\end{proof}

\section{Finite Direct Sums of Weakly IN (Strongly CS) Modules}

We begin by providing necessary and sufficient conditions for a finite direct sum of modules to be weakly IN.

\begin{theor}\label{weakly IN direct sum} Let $M=M_1\oplus M_2\oplus \cdots \oplus M_n$ be a direct sum of submodules
$M_i (1 \leq i \leq n)$. Then the following are equivalent:

\noindent{\em (1)} $M$ is weakly IN;

\noindent{\em (2)} {\em (a)}  $M_i$ is a weakly IN $R$-module for all $i\in \{1,\dots ,n\}$, and

{\em (b)}  $Ann_R(M_j)+Ann_R(M_k)=R$ for all $j\neq k \in \{1,\dots ,n\}$.
\end{theor}

\begin{proof} (1) $\Rightarrow$ (2) (a) follows from Proposition \ref{closed-submodules} and (b) follows from the definition of a weakly IN module.

(2) $\Rightarrow$ (1) Let $N$ and $L$ be two submodules of $M$ such that $N\cap L=0$. Using (b) and \cite[Lemma 2.6]{BO},
we get $N=\oplus _{i=1}^n(N\cap M_i)$ and $L=\oplus _{i=1}^n(L\cap M_i)$. Fix $i\in \{1,\dots ,n\}$.
As $N\cap L=0$, we have $(N\cap M_i)\cap (L\cap M_i)=0$. Since $M_i$ is weakly IN, we have $Ann_R(N\cap M_i)+Ann_R(L\cap M_i)=R$.
Moreover, using (b), it follows that for any $i\neq j \in \{1,\dots ,n\}$, we have $Ann_R(N\cap M_i)+Ann_R(L\cap M_j)=R$
as $Ann_R(M_i) \subseteq Ann_R(N\cap M_i)$ and $Ann_R(M_j) \subseteq Ann_R(L\cap M_j)$.
So $(\cap _{i=1}^nAnn_R(N\cap M_i)) + (\cap _{i=1}^nAnn_R(L\cap M_i))=R$ by \cite[Theorem 31]{Zariski et Sumuel:1979}.
Consequently, $Ann_R(N)+Ann_R(L)=R$. This completes the proof.
\end{proof}

In the next theorem, we provide a characterization of when a finite direct sum of modules is strongly CS. We need the following lemma.

\begin{lem}\label{product-CS-rings} Let $R_1$ and $R_2$ be two rings and let $M_i$ be an $R_i$-module $(i=1, 2)$.
Then the following hold true:

\noindent{\em (1)} $M_1\times M_2$ is a strongly CS $R_1\times R_2$-module if and only if $M_i$ is a strongly CS $R_i$-module for each $i=1, 2$.

\noindent{\em (2)} $R=R_1\times R_2$ is a CS ring if and only if so are $R_1$ and $R_2$.
\end{lem}

\begin{proof} (1) We will use the following elementary property: $(\ast)$ given an $R_1$-submodule $N_1$ of $M_1$ and an $R_2$-submodule $N_2$ of $M_2$,
$N_1\times N_2\subseteq ^{ess}M_1\times M_2$ if and only if $N_1\subseteq ^{ess}M_1$ and $N_2\subseteq ^{ess}M_2$.

Now suppose that $M_1\times M_2$ is a strongly CS $R_1\times R_2$-module and let $N_1$ be an $R_1$-submodule of $M_1$ and $N_2$ an $R_2$-submodule of $M_2$.
Then there exists $(e_1, e_2)=(e_1, e_2)^2=(e_1^2, e_2^2) \in R_1\times R_2$ such that $N_1\times N_2\subseteq ^{ess} (e_1,e_1)(M_1\times M_2)=
e_1M_1\times e_2M_2$. By $(\ast)$, it follows that $N_1\subseteq ^{ess}e_1M_1$ and $N_2\subseteq ^{ess}e_2M_2$.
This clearly implies that each $M_i$ ($i=1, 2$) is a strongly CS $R_i$-module. Conversely, let $N$ be an $R_1\times R_2$-submodule of $M_1\times M_2$.
Then $N=N_1\times N_2$, where $N_1$ is an $R_1$-submodule of $M_1$ and $N_2$ is an $R_2$-submodule of $M_2$.
Let $e_1^2=e_1\in R_1$ and $e_2^2=e_2\in R_2$ such that $N_1\subseteq ^{ess}e_1M_1$ and $N_2\subseteq ^{ess}e_2M_2$.
Again by $(\ast)$, we have $N_1\times N_2\subseteq ^{ess} e_1M_1\times e_2M_2=(e_1,e_1)(M_1\times M_2)$.
Since $(e_1,e_2)^2=(e_1,e_2)$, we conclude that $M_1\times M_2$ is a strongly CS $R_1\times R_2$-module.

(2) Apply (1) for $M_1=R_1$ and $M_2=R_2$ (see Remark \ref{submodules of WIN modules}(1)).
\end{proof}

\begin{theor}\label{finite strongly CS direct sum} Let $M=M_1\oplus M_2\oplus \cdots \oplus M_n$ be a direct sum of submodules
$M_i (1 \leq i \leq n)$. Then the following are equivalent:

\noindent{\em (1)} $M$ is strongly CS;

\noindent{\em (2)} $M$ satisfies the following two conditions:

\noindent{\em (a)} $M_i$ is a strongly CS $R$-module for every $i\in \{1,\dots ,n\}$, and

\noindent{\em (b)} There exists a complete set of orthogonal idempotents $\{e_1,\dots ,e_n\}$ of $R$ such that $e_iM_i=M_i$ for all $i\in \{1,\dots ,n\}$;

\noindent{\em (3)} $R=R_1\times R_2\times \cdots \times R_n$ such that $M_i$ is a strongly CS $R_i$-module for all $i\in \{1,\dots ,n\}$.
\end{theor}

\begin{proof} (1) $\Rightarrow$ (2) (a) follows by using Proposition \ref{closed-submodules}. Let us show (b) by induction on $n$.
Suppose that $M=M_1\oplus M_2$ is strongly CS. By Theorem \ref{WIN}, there exists an idempotent $e$ of $R$ such that $1-e\in Ann_R(M_1)$
and $e\in \Ann_R(M_2)$. This implies that $M_1=eM$ and $M_2=(1-e)M$. Therefore (b) is true for $n=2$. Now assume (b) holds for $n$;
we will prove it for $n+1$. Let $M=M_1\oplus \cdots \oplus M_n\oplus M_{n+1}$ be strongly CS. From the case $n=2$, we infer that
there exists an idempotent $e$ of $R$ such that $e(M_1\oplus \cdots \oplus M_n)=M_1\oplus \cdots \oplus M_n$ and $(1-e)M_{n+1}=M_{n+1}$.
But $M_1\oplus \cdots \oplus M_n$ is a strongly CS $R$-module as it is a submodule of $M$, so, by induction hypothesis there exists a complete set
of orthogonal idempotents $\{f_1,\dots ,f_n\}$ of $R$ such that $M_i=f_iM_i$ for all $i\in \{1,\dots ,n\}$.
It is easy to see that $\{ef_1,\dots ,ef_n,1-e\}$ is a complete set of orthogonal idempotents of $R$ such that
$ef_iM_i=eM_i=M_i$ for all $i\in \{1,\dots ,n\}$ and $(1-e)M_{n+1}=M_{n+1}$.

(2) $\Rightarrow$ (3) By (b), $R=R_1\times R_2\times \cdots \times R_n$ where $R_i=e_iR$ for all $i\in \{1,\dots ,n\}$. Fix $j \in \{1,\dots ,n\}$.
Since $e_jM_j=M_j$, $M_j$ has a natural structure of an $R_j$-module. Let $N$ be an $R_j$-submodule of $M_j$.
By (a), $M_j$ is a strongly CS $R$-module. So there exists $e^2=e\in R$ such that $N\subseteq^{ess}eM_j=ee_jM_j$ as $R$-modules and also as $e_jR$-modules.
Note that $ee_j$ is an idempotent of $R_j$. Thus $M_j$  is a strongly CS $R_j$-module.

(3) $\Rightarrow$ (1) This follows from Lemma \ref{product-CS-rings}(1).
\end{proof}

As an application of the preceding two theorems, we have the following corollary.

\begin{cor}\label{prime ideals}  Let $\mathfrak{p}$ and $\mathfrak{q}$ be two prime ideals of a ring $R$. Then the following hold true:

\noindent{\em (1)} The $R$-module $R/\mathfrak{p}\oplus R/\mathfrak{q}$ is weakly IN if and only if $\mathfrak{p}+\mathfrak{q}=R$.

\noindent{\em (2)} The following are equivalent:

\noindent{\em (a)} The $R$-module $R/\mathfrak{p}\oplus R/\mathfrak{q}$ is strongly CS;

\noindent{\em (b)} There exists an idempotent $e$ of $R$ such that $e\in \mathfrak{p}$ and $1-e\in \mathfrak{q}$;

\noindent{\em (c)} $\mathfrak{p}+\mathfrak{q}=R$ and idempotents lift modulo $\mathfrak{p}\cap\mathfrak{q}$.
\end{cor}

\begin{proof} Since the $R$-modules $R/\mathfrak{p}$ and $R/\mathfrak{q}$ are uniform, they are strongly CS (and hence also weakly IN).

(1) Use Theorem \ref{weakly IN direct sum} and the fact that $R/\mathfrak{p}$ and $R/\mathfrak{q}$ are weakly IN.

(2) (a) $\Leftrightarrow$ (c) Use (1), Theorem \ref{WIN} and the fact that $Ann_R(R/\mathfrak{p}\oplus R/\mathfrak{q})=\mathfrak{p}\cap\mathfrak{q}$.

(a) $\Rightarrow$ (b) By Theorem \ref{WIN}, there exists an idempotent $e$ of $R$ such that $e\in Ann_R(R/\mathfrak{p})=\mathfrak{p}$
and $1-e\in Ann_R(R/\mathfrak{q})=\mathfrak{q}$.

(b) $\Rightarrow$ (a) Let $e$ be an idempotent of $R$ such that $e\in \mathfrak{p}$ and $1-e\in \mathfrak{q}$.
Then $(1-e)(R/\mathfrak{p})=R/\mathfrak{p}$ and $e(R/\mathfrak{q})=R/\mathfrak{q}$.
Now using Theorem \ref{finite strongly CS direct sum}((2) $\Rightarrow$ (1)) and the fact that $R/\mathfrak{p}$ and $R/\mathfrak{q}$
are strongly CS $R$-modules, we deduce that $R/\mathfrak{p}\oplus R/\mathfrak{q}$ is a strongly CS $R$-module.
\end{proof}

In contrast to Example \ref{weakly IN-not-strongly CS}, we characterize in the next theorem the class of rings $R$ for which
every weakly IN $R$-module is strongly CS.

\begin{theor}\label{clean rings} The following are equivalent for a ring $R$:

\noindent{\em (1)} Any weakly IN $R$-module is strongly CS;

\noindent{\em (2)} $R/\mathfrak{m}\oplus R/\mathfrak{m}'$ is a strongly CS $R$-module for every distinct maximal ideals $\mathfrak{m}$ and $\mathfrak{m}'$ of $R$;

\noindent{\em (3)} Idempotents of $R$ lift modulo $\mathfrak{m}\cap \mathfrak{m}'$ for every distinct maximal ideals $\mathfrak{m}$ and $\mathfrak{m}'$ of $R$;

\noindent{\em (4)} Idempotents lift modulo every ideal of $R$;

\noindent{\em (5)} For every distinct maximal ideals $\mathfrak{m}$ and $\mathfrak{m}'$ of $R$, there exists an idempotent $e$ of $R$ such that $e\in \mathfrak{m}$ and $1-e\in \mathfrak{m}'$;

\noindent{\em (6)} $R$ is a clean ring.
\end{theor}

\begin{proof} The equivalences (4) $\Leftrightarrow$ (5)  $\Leftrightarrow$ (6) follow from \cite[Theorem 5.1]{AT}.

The equivalences (2) $\Leftrightarrow$ (3) $\Leftrightarrow$ (5) follow from Corollary \ref{prime ideals}(2).

(4) $\Rightarrow$ (1) Use the equivalence (1) $\Leftrightarrow$ (3) in Theorem \ref{WIN}.

(1) $\Rightarrow$ (2) Let $M=R/\mathfrak{m}\oplus R/\mathfrak{m}'$ where $\mathfrak{m}$ and $\mathfrak{m}'$ are two distinct maximal ideals of $R$.
By Corollary \ref{prime ideals}(1), $M$ is a weakly IN $R$-module and so it is strongly CS by (1).
\end{proof}

In the same vein of Example \ref{weakly IN-not-strongly CS}, we exhibit the following examples.

\begin{ex}\em (1) Let $R$ be a ring which is not clean. By Theorem \ref{clean rings}, there exists an $R$-module $M$ such that $M$ is weakly IN
but not strongly CS. To construct an explicit example of such a ring $R$ and such a module $M$,
consider the ring $R={\Bbb Z}$ and the $\mathbb{Z}$-module $M={\Bbb Z}/2{\Bbb Z}\oplus {\Bbb Z}/3{\Bbb Z}$.
Since $2{\Bbb Z}+3{\Bbb Z}={\Bbb Z}$, $M$ is a weakly IN ${\Bbb Z}$-module (Corollary \ref{prime ideals}(1)).
However, note that $M$ is not uniform. Then $M$ can not be a strongly CS $R$-module since $R$ is indecomposable
(Remark \ref{submodules of WIN modules}(3)).

(2) Let $R=\mathbb{Z}$ and consider the $R$-module $N=\mathbb{Q}/\mathbb{Z}_{(2\mathbb{Z})}$. By \cite[Example 3.13]{SW},
$A=R\propto N$ is a Torch ring. Thus, every cyclic $A$-module is weakly IN by Theorem \ref{cyclics are weakly IN}.
On the other hand, since $R$ is not a clean ring, $A$ is not clean by \cite[Theorem 6.4]{AW}.
Therefore the ring $A$ has a cyclic $A$-module which is not strongly CS by Theorem \ref{cyclics are strongly CS}.
\end{ex}

Recall that a ring $R$ is called {\it zero-dimensional} if every prime ideal of $R$ is maximal.

\begin{prop}\label{0-dimensional-rings} For a ring $R$ the following statements are equivalent:

\noindent{\em (1)} $R/\mathfrak{p}\oplus R/\mathfrak{p}'$ is a weakly IN $R$-module for every distinct prime ideals $\mathfrak{p}$ and $\mathfrak{p}'$ of $R$;

\noindent{\em (2)} $R/\mathfrak{p}\oplus R/\mathfrak{p}'$ is a strongly CS $R$-module for every distinct prime ideals $\mathfrak{p}$ and $\mathfrak{p}'$ of $R$;

\noindent{\em (3)} $R$ is a zero-dimensional ring.
\end{prop}

\begin{proof} (1) $\Rightarrow$ (3) Suppose that $R$ has a nonmaximal prime ideal $\mathfrak{p}$ and let $\mathfrak{m}$
be a maximal ideal containing $\mathfrak{p}$.
By hypothesis, $R/\mathfrak{p}\oplus R/\mathfrak{m}$ is a weakly IN $R$-module and hence $\mathfrak{p}+\mathfrak{m}=R$
(see Corollary \ref{prime ideals}(1)). Thus $\mathfrak{m}=R$, which is a contradiction. Therefore $R$ is a zero-dimensional ring.

(3) $\Rightarrow$ (2) Let $\mathfrak{p}$ and $\mathfrak{p}'$ be two distinct prime ideals of $R$. Since $R$ is zero-dimensional, $\mathfrak{p}$ and $\mathfrak{p}'$ are maximal. Moreover, note that $R$ is a clean ring by \cite[Corollary 11]{AC}. So $R/\mathfrak{p}\oplus R/\mathfrak{p}'$
is a strongly CS $R$-module by Theorem \ref{clean rings}.

(2) $\Rightarrow$ (1) This follows from the fact that any strongly CS $R$-module is weakly IN (see Theorem \ref{WIN}).
\end{proof}

Recall that a ring $R$ is called an {\it mp-ring} if every prime ideal contains a unique minimal prime ideal;
equivalently, every maximal ideal of $R$ contains a unique minimal prime ideal.

Replacing the term ``prime" in Proposition \ref{0-dimensional-rings} by ``minimal prime", we obtain the following characterizations.

\begin{prop}\label{mp-rings} The following are equivalent for a ring $R$:
\begin{enumerate}
\item[{\em (1)}] $R/\mathfrak{p}\oplus R/\mathfrak{p}'$ is a weakly IN $R$-module for every distinct minimal prime ideals
$\mathfrak{p}$ and $\mathfrak{p}'$ of $R$;

\item[{\em (2)}] $\mathfrak{p}+\mathfrak{p}'=R$ for every distinct minimal prime ideals $\mathfrak{p}$ and $\mathfrak{p}'$ of $R$;

\item[{\em (3)}] $R$ is an mp-ring.
\end{enumerate}
\end{prop}

\begin{proof} (1)  $\Leftrightarrow$ (2) Use Corollary \ref{prime ideals}(1).

(2)  $\Leftrightarrow$ (3) see \cite[Theorem 6.2]{AT}.
\end{proof}

Following \cite[Definition 8.1]{AT}, a ring $R$ is said to be a {\it purified} ring if for every distinct minimal prime ideals $\mathfrak{p}$ and $\mathfrak{p}'$
of $R$, there exists an idempotent $e$ of $R$ such that $e\in \mathfrak{p}$ and $1-e\in \mathfrak{p}'$. Note that every purified ring is an mp-ring.

\begin{theor} \label{purified-characterization} For a ring $R$ the following are equivalent:

\noindent{\em (1)} $R/\mathfrak{p}\oplus R/\mathfrak{p}'$ is a strongly CS $R$-module for every distinct minimal prime ideals $\mathfrak{p}$ and $\mathfrak{p}'$ of $R$;

\noindent{\em (2)} $R$ is an mp-ring and idempotents lift modulo $\mathfrak{p}\cap \mathfrak{p}'$ for every distinct minimal prime ideals $\mathfrak{p}$ and $\mathfrak{p}'$ of $R$;

\noindent{\em (3)} $R$ is a purified ring.
\end{theor}

\begin{proof} This follows by combining Corollary \ref{prime ideals}(2) with Proposition \ref{mp-rings}.
\end{proof}

\section{Modules over Dedekind Domains}

This short section is devoted to the description of the structure of both weakly IN and strongly CS modules over Dedekind domains. Recall that for an $R$-module $M$, $Ass(M)$ denotes the set of prime ideals of $R$ associated to $M$, that is, $Ass(M)=\{\mathfrak{p}\in Spec(R)$ $|$ $\mathfrak{p}=Ann_R(x)$ for some $0\neq x\in M\}$.

\begin{theor}\label{SSC over Dedekind domains} Let $R$ be a Dedekind domain with field of fractions $K$ and let $M$ be a nonzero $R$-module.
Then the following are equivalent:

\noindent{\em (1)} $M$ is a strongly CS $R$-module;

\noindent{\em (2)} $M$ is a uniform $R$-module;

\noindent{\em (3)} $M$ is isomorphic to an $R$-submodule of $K$ or there exists a maximal ideal $\mathfrak{p}$  of $R$ such that $M\cong E(R/\mathfrak{p})$
or $M\cong R/\mathfrak{p}^{n}$ for some positive integer $n$.
\end{theor}

\begin{proof} (1) $\Leftrightarrow$ (2) This follows from the fact that $R$ is indecomposable (see Remark \ref{submodules of WIN modules}(3)).

(2) $\Rightarrow$ (3) Let $M$ be a nonzero uniform $R$-module.
Since $R$ is noetherian, $E(M)\cong E(R/\mathfrak{p})$ for some prime ideal $\mathfrak{p}$ of $R$ (see \cite[Corollary of Theorem 2.32]{Sharp and Vamos:1971}).   Hence $M$ is isomorphic to a submodule of $E(R/\mathfrak{p})$ and $Ass(M)=\{\mathfrak{p}\}$. If $\mathfrak{p}=0$, then $M$ is isomorphic to an $R$-submodule of $E(R)\cong K$.
Now assume that $\mathfrak{p}\neq 0$. Then $\mathfrak{p}$ is a maximal ideal of $R$ as $R$ is a Dedekind domain.
Clearly, $E=E(R/\mathfrak{p})$ is a torsion $R$-module.
Moreover, $M$ is indecomposable since $E$ is uniform. Using \cite[Theorem 10]{Kaplansky},
we infer that $M\cong E(R/\mathfrak{p})$ or $M\cong R/\mathfrak{q}^n$ for some maximal ideal $\mathfrak{q}$ of $R$ and some positive integer $n$.
In the latter case we have $\mathfrak{q} \in Ass(M)$ and consequently $\mathfrak{q}=\mathfrak{p}$.

(3) $\Rightarrow$ (2) Let $\mathfrak{p}$ be a nonzero prime ideal of $R$.
It is clear that $K$ and $E(R/\mathfrak{p})$ are uniform $R$-modules.
Also, $R/\mathfrak{p}^n$ is uniform since it is a uniserial $R$-module by \cite[Lemma 6.8]{Sharp and Vamos:1971}.
\end{proof}

Let $R$ be a Dedekind domain with field of fractions $K$. Recall that an $R$-submodule $F$ of $K$ is called a fractional ideal of $R$
if there is a nonzero element $r$ of $R$ such that $rF \subseteq R$. Note that every finitely generated $R$-submodule of $K$ is a fractional
ideal of $R$ (see \cite[Lemma 6.15]{Sharp and Vamos:1971}). It is well known that every injective $R$-module has no maximal submodules.
The following corollary is an immediate consequence of the preceding theorem.

\begin{cor} \label{f.g. SSC over Dedekind domains} Let $R$ be a Dedekind domain and let $M$ be a nonzero finitely generated $R$-module.
Then the following are equivalent:

\noindent{\em (1)} $M$ is strongly CS;

\noindent{\em (2)} $M\cong I$ where $I$ is a nonzero fractional ideal of $R$ or there exists a maximal ideal $\mathfrak{p}$ of $R$ such that
$M\cong R/\mathfrak{p}^{n}$ for some positive integer $n$.
\end{cor}

\begin{theor}\label{WIN over Dedekind domains} Let $R$ be a Dedekind domain with field of fractions $K$ and let $M$ be a nonzero $R$-module.
Then the following are equivalent:

\noindent{\em (1)} $M$ is a weakly IN $R$-module;

\noindent{\em (2)} $M$ is isomorphic to an $R$-submodule of $K$ or $M\cong E(R/\mathfrak{p})$ for some maximal ideal $\mathfrak{p}$ of $R$
or $M\cong R/\mathfrak{p}_1^{n_1}\oplus \cdots \oplus R/\mathfrak{p}_k^{n_k}$ for some positive integers $n_i$ $(1 \leq i \leq k)$
and distinct maximal ideals $\mathfrak{p}_1,\dots , \mathfrak{p}_k$ of $R$.
\end{theor}

\begin{proof} (1) $\Rightarrow$ (2) Let $M$ be a nonzero weakly IN $R$-module. If $Ann_R(M)=0$, then $M$ is strongly CS
by Corollary \ref{Ann(R) direct summand}. From Theorem \ref{SSC over Dedekind domains}, it follows that $M$ is isomorphic to an $R$-submodule of $K$
or $M\cong E(R/\mathfrak{p})$ for some maximal ideal $\mathfrak{p}$ of $R$. Now suppose that $Ann_R(M)\neq 0$.
By \cite[Theorem 6.11]{Sharp and Vamos:1971}, there exists a family $\{\mathfrak{p}_i, n_i\}_{i\in I}$ such that $\mathfrak{p}_i$ are maximal ideals of $R$
and there are only finitely many distinct ones, $n_i$ $(i \in I)$ are positive integers and $M=\oplus _{i\in I}M_i$ where $M_i\cong R/\mathfrak{p}_i^{n_i}$ for every $i\in I$. Let $j\neq k\in I$. Since $M_j\cap M_k=0$ and $M$ is weakly IN, we have $Ann_R(M_j)+Ann_R(M_k)=R$. Thus, $\mathfrak{p}_j^{n_j}+\mathfrak{p}_k^{n_k}=R$. So $\mathfrak{p}_j\neq \mathfrak{p}_k$ for all $j\neq k$ in $I$. Consequently, $I$ is a finite set.

(2) $\Rightarrow$ (1) This follows from Theorems \ref{weakly IN direct sum} and \ref{SSC over Dedekind domains}.
\end{proof}

\begin{rem} \label{cyclic-weakly-IN} \rm Let $R$ be a Dedekind domain and let $I$ be a nonzero ideal of $R$. By \cite[Lemma 6.12 and Theorem 6.14]{Sharp and Vamos:1971},
$I=\mathfrak{p}_1^{n_1}\mathfrak{p}_2^{n_2}\dots \mathfrak{p}_k^{n_k}=\mathfrak{p}_1^{n_1}\cap \mathfrak{p}_2^{n_2}\cap \cdots \cap \mathfrak{p}_k^{n_k}$
for some distinct maximal ideals $\mathfrak{p}_1,\dots , \mathfrak{p}_k$ of $R$ and some positive integers $n_i$ $(1 \leq i \leq k)$.
By the Chinese Remainder Theorem, we have $R/I\cong R/\mathfrak{p}_1^{n_1}\oplus \cdots \oplus R/\mathfrak{p}_k^{n_k}$.
Now using Theorem \ref{WIN over Dedekind domains}, we conclude that every cyclic $R$-module is weakly IN (see also Theorem \ref{cyclics are weakly IN}).
\end{rem}

Combining Theorem \ref{WIN over Dedekind domains} and Remark \ref{cyclic-weakly-IN}, we obtain the following corollary.

\begin{cor} \label{f.g. WIN over Dedekind domains} Let $R$ be a Dedekind domain and let $M$ be a nonzero finitely generated $R$-module. Then the following are equivalent:

\noindent{\em (1)} $M$ is weakly IN;

\noindent{\em (2)} $M$ is cyclic or $M\cong I$ where $I$ is a nonzero fractional ideal of $R$.
\end{cor}

\section{CS Trivial Extensions}

Let $R$ be a ring and let $M$ be an $R$-module. The abelian group $R\oplus M$ can be endowed with the following product: $(a,x)(b,y)=(ab,ay+bx)$. The result is a ring called the trivial extension of $R$ by $M$ denoted by $R\varpropto M$. $R$ becomes a subring of $R\varpropto M$ and $M$ an ideal such that $M^2=0$.
If $I$ is an ideal of $R$ and $N$ is a submodule of $M$ such that $IM\subseteq N$, then $(I,N)=\{(a,x)\in R\varpropto M \mid a\in I,\ x\in N\}$ is an ideal of $A=R\varpropto M$ and we have $Ann_A(I,N)=(Ann_R(I)\cap Ann_R(N),Ann_M(I)$).
In this section our main result is a characterization of CS trivial extensions. To prove it, we need the following three lemmas.

\begin{lem}\label{T-Ext-M-weakly-IN} Let $R$ be a ring and let $M$ be an $R$-module such that $A=R\propto M$ is a CS ring.
Then $M$ is a weakly IN $R$-module.
\end{lem}

\begin{proof} Let $N$ and $L$ be two submodules of  $M$ such that $N\cap L=0$.
Then $(0,N)$ and $(0,L)$ are two ideals of $A$ satisfying $(0,N)\cap (0,L)=0$. By Proposition \ref{CS rings}, $Ann_A(0,N)+Ann_A(0,L)=A$.
It follows that $(Ann_R(N),M)+(Ann_R(L),M)=A$ and hence $Ann_R(N)+Ann_R(L)=R$. Therefore $M$ is weakly IN.
\end{proof}

\begin{lem}\label{faithful CS TX} Let $R$ be a ring and let $M$ be a faithful $R$-module. Then the following are equivalent:

\noindent{\em (1)} $A=R\propto M$ is a CS ring;

\noindent{\em (2)} $M$ is a weakly IN $R$-module;

\noindent{\em (3)} $M$ is a strongly CS $R$-module.
\end{lem}

\begin{proof} (1) $\Rightarrow$ (2) See Lemma \ref{T-Ext-M-weakly-IN}.

(2) $\Leftrightarrow$ (3) This follows from Corollary \ref{Ann(R) direct summand}.

(3) $\Rightarrow$ (1) Let $I$ be an ideal of $A$ and let $V=\{x\in M\ |\ (0,x)\in I\}$.  Then $V$ is a submodule of $M$ and $I\cap (0,M)=(0,V)$.
Since $M$ is strongly CS, there exists $e=e^2\in R$ such that $V\subseteq ^{ess}eM$. First let us show that $I\subseteq (e,0)A=(eR,eM)$.
Consider an element $(a,x)$ of $I$. Then, for any element $z$ of $M$, $(a,x)(0,z)=(0,az)\in I\cap (0,M)=(0,V)$.
Thus $az\in V$ and hence $aM\subseteq V\subseteq eM$. Consequently, $(1-e)aM=0$. But $M$ is a faithful $R$-module, so $(1-e)a=0$ and hence $a=ea\in eR$. Moreover, $(a,x)(1-e,0)=(a(1-e),(1-e)x)=(0,(1-e)x)\in I\cap (0,M)=(0,V)$. This gives $(1-e)x\in V\subseteq eM$. Thus $x-ex\in eM$ and so $x\in eM$.
Therefore $I\subseteq (eR,eM)$. Now let us show that $I\subseteq ^{ess} (eR,eM)$. Let $(0, 0)\neq (ea,ex)\in (eR,eM)$, where $a\in R$ and $x\in M$.
If $ea\neq 0$, then $eaM\neq 0$ since $M$ is faithful. So there exists $y\in M$ such that $eay\neq 0$. But $0\neq eay\in eM$ and  $V\subseteq ^{ess}eM$,
so there exists $t\in R$ such that $0\neq teay \in V$. Hence $(0, 0)\neq (0,ty)(ea,ex)=(0,teay)\in (0,V)=I\cap (0,M)$. Therefore $(0, 0)\neq (0,ty)(ea,ex)\in I$. Now suppose that $ea=0$. Then $ex\neq 0$. Since $V\subseteq ^{ess}eM$, there exists $u\in R$ such that $0\neq uex \in V$.
So $(0, 0)\neq (u,0)(ea,ex)=(0,uex)\in (0,V)=I\cap (0,M)$. Hence  $0\neq (u,0)(ea,ex)\in I$. It follows that $I\subseteq ^{ess} (e,0)A$.
Note that $(e,0)^2=(e,0)$. Consequently, $A$ is a CS ring.
\end{proof}

\begin{lem}\label{product of weakly IN rings} Let $R$ be a ring and let $M$ be a nonzero $R$-module. Let $e^2=e\in R$.
Then the following hold true:

\noindent{\em (1)} $eR$ is strongly CS as $R$-module if and only if $eR$ is a CS ring.

\noindent{\em (2)} Assume that $Ann_R(M)=eR$. Then $(1-e)M$ is a faithful $(1-e)R$-module.
Moreover, the rings $A=R\propto M$ and $eR\times ((1-e)R\propto (1-e)M)$ are isomorphic.
\end{lem}

\begin{proof} (1) Note the $R$-submodules and the $eR$-submodules of $eR$ are the same.

($\Rightarrow$) Let $I$ be an ideal of the ring $eR$. Since $eR$ is a strongly CS $R$-module, there exists $f=f^2\in R$
such that $I$ is an essential $R$-submodule of $feR$. Thus $I$ is an essential $eR$-submodule of $(fe)eR$.
Moreover, $fe=(fe)^2\in eR$. Therefore $eR$ is a CS ring.

($\Leftarrow$) Let $I$ be an ideal of $R$ contained in $eR$. Then $I=eI$ is an ideal of $eR$.  Since $eR$ is a CS ring, there exists an idempotent $f\in eR$
such that $I$ is an essential $eR$-submodule of $feR$. It is clear that $I$ is also an essential $R$-submodule of $feR$.

(2) It is clear that $R=eR\oplus (1-e)R$ and $M=eM\oplus (1-e)M$. By \cite[Theorem 4.4]{AW}, $A=R\propto M\cong (eR\propto eM)\times ((1-e)R\propto (1-e)M)$
(as rings). But $eM=0$, so the ring $A$ is isomorphic to  the ring $eR\times ((1-e)R\propto (1-e)M)$.  The first assertion is obvious.
\end{proof}

\begin{theor}\label{CS TX} The following are equivalent for a ring $R$ and an $R$-module $M$:

\noindent{\em (1)} $A=R\propto M$ is a CS ring;

\noindent{\em (2)} $M$ satisfies the following two conditions:

\noindent{\em (a)} $Ann_R(M)$ is a direct summand of $R$ which is a CS ring, and

\noindent{\em (b)} $M$ is a weakly IN $R$-module;

\noindent{\em (3)} $M$ satisfies the following two conditions:

\noindent{\em (a)} $Ann_R(M)$ is a direct summand of $R$ which is a CS ring, and

\noindent{\em (b)} $M$ is a strongly CS $R$-module;

\noindent{\em (4)} $M\oplus Ann_R(M)$ is a weakly IN $R$-module;

\noindent{\em (5)} $M\oplus Ann_R(M)$ is a strongly CS $R$-module.
\end{theor}

\begin{proof} It is easily seen that $\{(e, 0) \in A \mid e^2=e \in R\}$ is the set of idempotents of $A$.

(1) $\Rightarrow$ (2) (a) Suppose that $Ann_R(M)\neq 0$. Clearly, $(Ann_R(M),0)$ is an ideal of $A$.
Since $A$ is a CS ring, there exists an idempotent $e$ of $R$ such that $(Ann_R(M),0)\subseteq ^{ess}(e,0)A=(eR,eM)$. Hence, $Ann_R(M)\subseteq eR$.
We claim that $eM=0$. Suppose, on the contrary, that $eM\neq 0$. Then $(0,eM)$ is a nonzero ideal of $A$ which is contained in $(eR,eM)$.
Thus $(0,eM)\cap (Ann_R(M),0)\neq 0$, a contradiction.
Therefore $eR\subseteq Ann_R(M)$. It follows that $Ann_R(M)=eR$ is a direct summand of $R$.
Moreover, $eR$ is a CS ring by Lemmas \ref{product-CS-rings}(2) and \ref{product of weakly IN rings}(2).

(b) follows from Lemma \ref{T-Ext-M-weakly-IN}.

(2) $\Rightarrow$ (3) Use Corollary \ref{Ann(R) direct summand}.

(3) $\Rightarrow$ (5) By hypothesis, there exists $e^2=e\in R$ such that $Ann_R(M)=eR$. Thus $Ann_R(M)=eAnn_R(M)$.
Moreover, since $R=eR\oplus (1-e)R$, we have $M=(1-e)M$. Also, $M$ and $Ann_R(M)$ are strongly CS $R$-modules by Lemma \ref{product of weakly IN rings}(1).
So, by Theorem \ref{finite strongly CS direct sum}, $M\oplus Ann_R(M)$ is a strongly CS $R$-module.

(5) $\Rightarrow$ (4) This follows from Theorem \ref{WIN}.

(4) $\Rightarrow$ (1) By Theorem \ref{weakly IN direct sum}, $Ann_R(M)+Ann_R(Ann_R(M))=R$.
Then there exist $a\in Ann_R(M)$ and $b\in Ann_R(Ann_R(M))$ such that $a+b=1$. Let $c\in Ann_R(M)\cap Ann_R(Ann_R(M))$.
Thus $c=ca+cb=0$ and hence $Ann_R(M)\cap Ann_R(Ann_R(M))=0$. It follows that $Ann_R(M)$ is a direct summand of $R$.
So $Ann_R(M)=eR$ for some idempotent $e$ of $R$. Note that $eR$ is a weakly IN $R$-module by Proposition \ref{closed-submodules}.
Moreover, the $R$-submodules and the $eR$-submodules of $eR$ are exactly the same. Also, we have $Ann_{eR}(L) = eAnn_R(L)$
for any $R$-submodule $L$ of $eR$. We thus deduce that $eR$ is a CS ring by Proposition \ref{CS rings}.
In addition, using Proposition \ref{closed-submodules} again, we see that $(1-e)M$ is a weakly IN $(1-e)R$-module.
Note that $(1-e)M$ is a faithful $(1-e)R$-module (see Lemma \ref{product of weakly IN rings}(2)).
Then $(1-e)R\propto (1-e)M$ is a CS ring by Lemma \ref{faithful CS TX}.
Since the rings $R\propto M$ and $eR\times ((1-e)R\propto (1-e)M)$ are isomorphic (see Lemma \ref{product of weakly IN rings}(2)), it follows that
$R\propto M$ is a CS ring by Lemma \ref{product-CS-rings}(2).
\end{proof}

\begin{rem} \em Let $R$ be a ring and let $M$ be an $R$-module such that $R\propto M$ is a CS ring.
By Theorem \ref{CS TX}, $M$ is a strongly CS $R$-module. However, the ring $R$ need not be a CS ring, in general, as it is shown in the following example.
Let $R$ be a local ring which not uniform (for example, we can take $R=K\propto (K\oplus K)$, where $K$ is a field)
and let ${\mathfrak m}$ be the maximal ideal of $R$. Consider the $R$-module $M=E(R/{\mathfrak m})$ and the ring $A=R\propto M$.
By \cite[Corollary 2 of Proposition 2.26]{Sharp and Vamos:1971}, $M$ is a faithful $R$-module. Moreover, since $M$ is a uniform $R$-module,
$M$ is strongly CS. Thus $A$ is a CS ring by Lemma \ref{faithful CS TX} (indeed, $A$ is uniform).
Since $R$ is local, $R$ is indecomposable and so $R$ cannot be a CS ring since it is not uniform (see Remark \ref{submodules of WIN modules}(3)).
However, when the $R$-module $M$ is flat, then the condition $R\propto M$ is CS forces $R$ to become CS as it is shown in the following result.
\end{rem}

\begin{cor}\label{flat} Let $R$ be a ring and let $M$ be a flat $R$-module. Then the following are equivalent:

\noindent{\em (1)} $A=R\propto M$ is a CS ring;

\noindent{\em (2)} $R$ and $M$ are weakly IN $R$-modules and $Ann_R(M)$ is a direct summand of $R$;

\noindent{\em (3)} $R$ and $M$ are strongly CS $R$-modules and $Ann_R(M)$ is a direct summand of $R$.
\end{cor}

\begin{proof} (1) $\Rightarrow$ (3) By Theorem \ref{CS TX}, $M$ is a strongly CS $R$-module and $Ann_R(M)$ is a direct summand of $R$.
To show that $R$ is CS, take two ideals $I$ and $J$ of $R$ such that $I\cap J=0$. Since $M$ is a flat $R$-module, we have $IM\cap JM=(I\cap J)M$ by \cite[Proposition 8.5]{Tuganbaev}. Thus $IM\cap JM=0$. Consider the ideals $I'=(I,IM)$ and $J'=(J,JM)$ of $A$. Since $I'\cap J'=0$ and $A$ is CS, we have   $Ann_A(I')+Ann_A(J')=A$ (see Proposition \ref{CS rings}). It follows that $$(Ann_R(I),Ann_M(I))+(Ann_R(J),Ann_M(J))=A.$$
So $Ann_R(I)+Ann_R(J)=R$. Using again Proposition \ref{CS rings}, we conclude that $R$ is a CS ring. Hence $R$ is strongly CS.

(3) $\Rightarrow$ (2) This follows from Theorem \ref{WIN}.

(2) $\Rightarrow$ (1) Since $R$ is weakly IN, $R$ is a strongly CS ring by Remark \ref{submodules of WIN modules}(1). As $Ann_R(M)$ is a direct summand of $R$,
it follows that $Ann_R(M)$ is a CS ring (Proposition \ref{closed-submodules} and Lemma \ref{product of weakly IN rings}(1)). Now use Theorem \ref{CS TX}.
\end{proof}

\begin{cor} Let $H$ be a faithful ideal of $R$ and let $T(R)$ be the total ring of fractions of $R$. Then the following are equivalent:

\noindent{\em (1)} $R\propto H$ is a CS ring;

\noindent{\em (2)} $R\propto T(R)$ is a CS ring;

\noindent{\em (3)} $R$ is a CS ring.
\end{cor}
\begin{proof} Note that $T(R)$ is a flat $R$-module (see \cite[Corollary 3.6]{AM}).

(1) $\Rightarrow$ (3) Let $I$ and $J$ be two ideals of $R$ such that $I\cap J=0$. So $IH\cap JH=0$. But $IH$ and $JH$ are two $R$-submodules
of the faithful $R$-module $H$. Moreover, $H$ is weakly IN $R$-module by Lemma \ref{faithful CS TX}. Therefore $Ann_R(IH)+Ann_R(JH)=R$.
But $Ann_R(IH)=Ann_R(I)$ and $Ann_R(JH)=Ann_R(J)$ as $H$ is faithful. Then $Ann_R(I)+Ann_R(J)=R$. By Proposition \ref{CS rings},
we see that $R$ is a CS ring.

(3) $\Rightarrow$ (2) It is clear that $Ann_R(T(R))=0$. To prove that $R\propto T(R)$ is a CS ring, we only need to show that
$T(R)$ is a strongly CS $R$-module (see Corollary \ref{flat}).
Let $S$ denote the set of all non-zero-divisors in $R$. Then $T(R)=S^{-1}R$. Let $I'$ be an ideal of $T(R)$.
Then $I'=S^{-1}I$ for some ideal $I$ of $R$. But $R$ is a CS ring, so there exists $e^2=e\in R$ such that $I\subseteq ^{ess} eR$.
It is easily seen that $I'=S^{-1}I\subseteq ^{ess} eT(R)$.
Consequently, $T(R)$ is a strongly CS $R$-module.

(2) $\Rightarrow$ (1) By Corollary \ref{flat}, $R$ is a weakly IN $R$-module.
Since $H$ is an $R$-submodule of $R$, $H$ is also a weakly IN $R$-module. According to Lemma \ref{faithful CS TX}, it follows that
$R\propto H$ is a CS ring.
\end{proof}


\begin{thebibliography}{MM}

\bibitem{AT} M. Aghajani and A. Tarizadeh, Characterizations of Gelfand rings, specially clean rings and their dual rings,
Results Math. {\bf 75}, 125 (2020).

\bibitem{AF} F. W. Anderson and K. R. Fuller, Rings and Categories of Modules, 2nd ed., Graduate Texts in Mathematics, vol. {\bf 13},
Springer-Verlag, New York, 1992.

\bibitem{AC} D. D. Anderson and V. P. Camillo, Commutative rings whose elements are a sum of a unit and an idempotent,
Comm. Algebra {\bf 30}(7) (2002) 3327-3336.

\bibitem{AW} D. D. Anderson and M. Winders, Idealization of a module, J. Commut. Algebra {\bf 1}(1) (2009) 3-56.

\bibitem{AM} M. F. Atiyah and I. G. Macdonald, Introduction to Commutative Algebra, Addison-Wesley Publishing Company, London, 1969.

\bibitem{BO} Z. Bilgin and K. H. Oral, Coprimely structured modules, Palestine J. Math. {\bf 7}(Special Issue: I) (2018) 161-169.

\bibitem{Brandal} W. Brandal, Commutative Rings whose Finitely Generated Modules Decompose, Lecture Notes in Mathematics, vol. {\bf 723}, Springer-Verlag, Berlin, 1979.

\bibitem{CNY}  V. Camillo, W. K. Nicholson and M. F. Yousif, Ikeda-Nakayama rings, J. Algebra {\bf 226}(2) (2000) 1001-1010.

\bibitem{DHSW} N. V. Dung, D. V. Huynh, P. F. Smith and R. Wisbauer, Extending Modules, Pitman Research Notes in Mathematics Series, {\bf 313}, Longman Scientific \& Technical, Harlow, 1994.

\bibitem{Kaplansky} I. Kaplansky, Modules over Dedekind rings and valuation rings, Trans. Amer. Math. Soc. {\bf 72} (1952) 327-340.

\bibitem{Kourki:2003} F. Kourki, On some annihilator conditions over commutative rings, Lecture Notes Pure Appl. Math. {\bf 231}
(2003) 323-334.

\bibitem{MM} S. H. Mohamed and B. J. M\"{u}ller, Continuous and Discrete Modules, London Math. Soc. Lecture Note Series, vol. {\bf 147}, Cambridge University Press, Cambridge, 1990.

\bibitem{NY} W. K. Nicholson and M. F. Yousif, Quasi Frobenius Rings, Cambridge Tracts in Mathematics {\bf 158}, Cambridge University Press, 2003.

\bibitem{Sharp and Vamos:1971} D. W. Sharp and P. Vamos, Injective modules, Cambridge Tracts in Mathematics and Mathematical Physics, No. {\bf 62}, Cambridge University Press, 1971.

\bibitem{SW} T. S. Shores and R. Wiegand, Rings whose finitely generated modules are direct sums of cyclics, J. Algebra {\bf 32} (1974) 152-172.

\bibitem{Tuganbaev} A. Tuganbaev, Rings close to regular, Mathematics and its Applications, {\bf 545}, Kluwer Academic Publishers, Dordrecht, 2002.

\bibitem{Wisbrauer et Yousif et Zhou :1994} R. Wisbrauer, M. F. Yousif and Y. Zhou, Ikeda-Nakayama Modules, Contributions to Algebra and Geometry
{\bf 43}(1) (2002) 111-119.

\bibitem{YoZh} M. F. Yousif and Y. Zhou, Semiregular, semiperfect and perfect rings relative to an ideal, Rocky Mountain J. Math.
{\bf 32}(4) (2002) 1651-1671.

\bibitem{Zariski et Sumuel:1979} O. Zariski and P. Samuel, Commutative Algebra, vol. I, Graduate Texts in Mathematics {\bf 28}, Springer-Verlag,
New York, 1979.

\end{thebibliography}
\end{document}